\documentclass[11pt]{article}
\usepackage{amsmath}
\usepackage{amssymb}
\usepackage{amsthm}
\usepackage{mathrsfs}
\usepackage{anysize}
\usepackage{natbib}
\usepackage{xspace}
\usepackage{setspace}

\usepackage{booktabs}
\usepackage{pdflscape}
\usepackage{dsfont}
\usepackage{enumerate}
\usepackage{graphicx}
\usepackage{floatrow}
\usepackage{bm}
\usepackage{multicol}
\usepackage{multirow}
\usepackage[dvips]{epsfig}

\usepackage{float}
\usepackage{color}
\usepackage[latin1]{inputenc}
\usepackage[english]{babel}
\usepackage{xcolor}

\usepackage{bbm}
\usepackage{mathtools}

\numberwithin{equation}{section}

\theoremstyle{plain}
\newtheorem{thm}{Theorem}
\newtheorem{cor}{Corollary}

\newtheorem{prop}{Proposition}
\theoremstyle{definition}

\theoremstyle{definition}

\theoremstyle{remark}

\marginsize{25mm}{25mm}{25mm}{25mm}

\newcommand{\I}{\mathds{1}}

\renewcommand\S{\mathcal S}

\renewcommand\L{\mathcal L}

\begin{document}

\title{On the convolution equivalence of tempered stable distributions on the real line\footnote{The author gratefully acknowledges the comments of Claudia Kl\"uppelberg, Pasquale Cirillo, and Iosif Pinelis, and the contribution to Proposition \ref{lemma} of Iosif Pinelis. They are not responsible of any errors.}}
\author{Lorenzo Torricelli\footnote{University of Bologna, Department of Statistical Sciences ``P. Fortunati''. Email: lorenzo.torricelli2@unibo.it}  }
\date{\today}

\maketitle

\begin{abstract}
We show the convolution equivalence property of univariate tempered stable distributions in the sense of \cite{ros:07}. This makes rigorous various classic heuristic arguments on the asymptotic similarity between the probability and L\'evy densities of such distributions. Some specific examples from the literature are discussed.
\end{abstract}


\noindent {\bf{AMS}}: 60E07


\section{Tempered stable distributions and tail properties}

We discuss here some tail properties of a certain  infinitely divisible (i.d.) distribution class on the real line called tempered stable distributions.
 For a distribution $\mu$ on $\mathbb R$  denote by $F$  (or $F_\mu$ if necessary)  its cumulative distribution function (c.d.f.) $F(x)=\mu((-\infty,x])$ and  by $\overline F(x)=1-F(x)=\mu((x, \infty))$ (resp. $\overline {F_\mu}$) its survival function (s.f.). Let further
$\hat \mu(z)=\int_{(-\infty,\infty)} e^{zx}\mu(dx)$  for all $z \in \mathbb R$ such that the integral converges, be the moment generating function (m.g.f.) of $\mu$, i.e. $\hat \mu(z) = \mathscr{L}(-z;\mu)$, where $\mathscr{L}(\cdot;\mu)$ indicates the Laplace transform of the measure $\mu$.

Saying that $\mu$ is \emph{infinitely divisible} (i.d.) amounts to the property that its characteristic function (c.f.) has the so-called L\'evy-Khintchine representation
\begin{equation}\label{eq:LK}
\int_{-\infty}^\infty e^{izx }\mu(dx)=  \exp\left(i  z  b  -  z^2 \sigma /2+\int_{\mathbb R} \left( e^{i z x} -1 -   i    z x\I_{\{| x |<1\}}\right)\nu( d x)\right), \qquad z \in \mathbb R
\end{equation}
 with $b \in \mathbb R$, $\sigma \geq 0$ and  $\nu$ is a measure on $\mathbb R$ such that $\nu(\{0 \})=0$, $\int_{\mathbb R} (x ^2 \wedge 1)\nu( dx) <\infty$, called the L\'evy measure.  With abuse of notation, if $\nu$ is absolutely continuous  we indicate a L\'evy density using the same expression as the corresponding measure, and always understand $\nu(\{0 \})=0$.

Classically, ``tempering'' means  tilting the c.d.f. or the probability density function (p.d.f.) of a distribution  by multiplication with an exponential of negative argument.  In their pioneering works, \citet{man+sta:95} and \citet{kop:95} applied tempering to L\'evy measures as well. In particular, the approach followed by the latter of applying an exponential smoothing factor to the stable L\'evy density proved to be particularly consequential, since  it solves the the infinite variance issue of stable random walk models in physics, while retaining analytical expressions for the c.f.s. It was later established in \citet{ros:07} that such approach  could be extended from exponential to general completely monotone tempering functions, and that multivariate versions of the procedure can be envisaged as well.

 We recall that a \emph{completely monotone} (c.m.) function $f: (0,\infty) \rightarrow \mathbb R$ is an infinitely-differentiable function whose derivatives satisfy
\begin{equation}\label{eq:cm}
(-1)^{n}f^{(n)}(x) \geq 0, \qquad x \geq 0, n \in \mathbb N.
\end{equation}
The limit $\lim_{ x \rightarrow 0} f(x)$ always exists and if it is finite \eqref{eq:cm} holds for the extension of $f$ on $[0,\infty)$ as well.    An important theorem by S. N. Bernstein states that $f$ is c.m. if and only if
\begin{equation}\label{eq:bernst}
f(x)=\mathscr{L}(x;\eta)
\end{equation}
for some positively-supported positive Borel measure $\eta$.

 Let $q_+,q_-:(0,\infty) \rightarrow \mathbb R$ be c.m. functions with corresponding Bernstein representing measures $Q_\pm$ in \eqref{eq:bernst} such that $\lim_{x \rightarrow \infty}q_\pm(x)=0$ and let

\begin{equation}\label{eq:rosrep}
\nu(x)=\delta_+\frac{q_+(x)}{x^{1+\alpha}}\I_{\left\{x>0\right\}}+\delta_-\frac{q_-(|x|)}{|x|^{1+\alpha}}\I_{\left\{x<0\right\}}, \qquad \delta_+, \delta_- \geq 0, \, \alpha \in (0,2).
\end{equation}
A \emph{tempered $\alpha$-stable}  (TS$_\alpha$) law $\mu$ on $\mathbb R$ in the sense of \citet{ros:07}  is an i.d. distribution with  L\'evy triplet $(b, 0, \nu)$ with $\nu$ given by \eqref{eq:rosrep}.  
If additionally $\lim_{x \rightarrow 0}q_\pm(x)=1$ then we extend $q_\pm$ on $[0, \infty)$ and say that $\mu$ is \emph{proper}. In other words a proper TS$_\alpha$ law is one for which $q_\pm$ are  probability s.f.s  of some probability distribution on $\mathbb R_+$, and the corresponding Bernstein representing measures $Q_\pm$ in \eqref{eq:bernst} are positive p.d.f.s. 

\smallskip

In the applied literature it is often stated or implied (e.g. \citet{kop:95}) that asymptotically the tails of i.d. p.d.f.s ``behave as'' those of their L\'evy densities.   
A first indication of some kind of similarity between  $\overline \mu(x)$ and those of $\overline \nu(x)$ for large $x$ is provided by the fact that  absolute/exponential moments of $\mu$ are finite if and only if those of $\nu\I_{\{x>1\}}/\bar \nu(1)$ are (e.g. \citet{sat:99}, Corollary 25.8). 
For instance  if $q_\pm(x)=e^{-\lambda x}$ , it follows from \eqref{eq:rosrep} that exponential moments  of order $\theta< \lambda$ exist, and thus all moments also do. In applications, this is precisely the idea motivating L\'evy tempering as a solution for resolving issues related to non-existence of moments.
As observed by \citet{wat+yam:10},  so long as the ratio $\overline F(x)/\overline \nu(x)$ is bounded from above and away from zero the tails of probability densities of an i.d. distribution correspond to those of the L\'evy densities in a weak sense i.e. setting $0<c=\liminf_{x \rightarrow \infty}F(x)/\overline \nu(x) $, $0<C=\limsup_{x \rightarrow \infty} F(x)/\overline \nu(x)$, then  for all $\delta \in (0,1)$ there exists $x_\delta>0$ such that
 \begin{equation}
c(1-\delta)  \overline \nu(x) < \overline F(x) < C(1+\delta) \overline \nu(x)
 \end{equation}
for all $x>x_\delta$. However, in general $c \neq C$, and the stronger statement 
 $\overline F(x) \sim D \, \overline \nu(x)$, $D=c=C$ does not need to hold. 
A known sufficient condition for this asymptotic equivalence to hold is the so called convolution equivalence property of the given probability distribution.

For $\gamma \geq 0$, a distribution $\mu$ on $\mathbb R$ is said to belong to the class $\mathcal L(\gamma)$   if for all for all $x,y \in \mathbb R$ we have $\overline F(x)>0$ and 
 \begin{equation}\label{eq:convlclass}\frac{\overline {F}(x+y)}{ \overline F(x)}=e^{- y \gamma}, \qquad x \rightarrow + \infty .\end{equation} It is instead said to be  \emph{convolution equivalent}, or of class $\mathcal S(\gamma)$, if for some $\gamma$ we have $\mu \in \mathcal L(\gamma)$ and

\begin{equation}\label{eq:yammain}
\lim_{x\rightarrow \infty}\frac{\overline{F*F}(x)}{\overline{F}(x)}= 2 \hat \mu (  \gamma)<\infty
\end{equation}
(i.e. \citet{wat:08}, Equations (1.1) and (1.2), \citet{wat+yam:10}, Definition 1.1).
In general it holds $\mathcal S(\gamma) \subsetneq  \mathcal L(\gamma)$ for all $\gamma\geq 0$. For example an Exp$(\lambda)$ distribution   belongs to  $\L(\lambda)$ but not to $\S(\lambda)$  since $\hat \mu(\lambda)=\infty$ and the limit \eqref{eq:yammain} is infinite. For $\gamma=0$,  examples in $\L(0) \setminus \S(0)$ are given in  \citet{pitman1980subexponential} and \citet{embrechts1980closure}. The class $\L(0)$ is called that of the \emph{long-tailed} distributions and $\mathcal S(0)$ that of the \emph{subexponential} distributions. It is easy to show that a long-tailed distribution is a particular instance of a heavy-tailed distribution -- the latter meaning that no exponential moment exists -- but these two classes do not coincide. Qualitatively, subexponentiality  amounts to the property that  the sum of two independent random variables distributes asymptotically as their maximum.

For full details and properties the classes $\L(\gamma)$ and $\S(\gamma)$ subexponentiality, long-tailedness, related concepts and applications we refer the reader to  \citet{embrechts1980closure}, \citet{embrechts1982convolution}, \citet{embrechts2013modelling}, \citet{cline1986convolution}, \citet{pakes2004convolution}, \citet{wat:08}, \citet{wat+yam:10} and references therein.

\smallskip

We show in this letter that TS$_\alpha$ laws are convolution equivalent. The consequence for applications is that their probability and  L\'evy densities are asymptotically equivalent, and the proportionality constant can be explicitly determined, so long as the c. f. of the model is known.

\section{Convolution equivalence of tempered stable distributions}\label{sec:dens}


The following Proposition is of independent interest and provides some noticeable properties of c.m s.f.s.


\begin{prop}\label{lemma}
Let $\mu$ be a probability distribution on $\mathbb R_+$ such that $\overline {F_\mu}$ is c.m. with  Bernstein representing probability measure $\eta$. Then:
\begin{enumerate}
 \item[(i)] $ \mu \in \L(\gamma)$ with $\gamma = \inf_{[0,\infty)} \mbox{{\upshape supp}}(\eta)$;
\item[(ii)]  the convergence in   \eqref{eq:convlclass} is monotone increasing for all $y \geq  0$.
\end{enumerate}  \end{prop}

\begin{proof}
We drop the subscript from the s.f..
W.l.o.g. we can assume $y >0$. Let $h>0$ be fixed and set $I_0(x)=\int_{[\gamma,\gamma+h)}e^{-s x} \eta(ds)$, $I_\infty(x)=\int_{[\gamma+h,\infty)}e^{-s x} \eta(ds)$. We have \begin{equation}\label{eq:alphaineq}
\frac{ F(x)}{F(x+y)} <\frac{I_0(x)}{I_{0}(x+y)}\left(1 +\frac{I_\infty(x)}{I_{0}(x)}\right).
\end{equation}
Now, letting $c(h):=\eta( [\gamma, \gamma+h/2))$
\begin{equation}
I_0(x) \geq \int_{ [\gamma, \gamma+h/2)} e^{-s x} \eta(ds) > c(h) e^{-x(\gamma+ h/2)}
\end{equation}
and $c(h)>0$ because $\gamma \in$ supp$(\eta)$. 
Furthermore
\begin{equation}
I_\infty(x) < e^{-x(\gamma+ h)}\eta( [\gamma+h,\infty)) <e^{-x (\gamma+ h)}
\end{equation}
since $\eta$ is a probability measure, so that 
\begin{equation}\label{eq:eps}
\frac{I_\infty(x)}{I_{0}(x)} < \frac{e^{-xh/2}}{c(h)}.
\end{equation} 
Moreover
\begin{equation}\label{eq:Iest}
I_0(x+y) > e^{-y (\gamma+ h)} I_0(x)
\end{equation}
and using \eqref{eq:eps} and \eqref{eq:Iest} in \eqref{eq:alphaineq} we obtain 
\begin{equation}
\frac{ F(x)}{F(x+y)} <  e^{y (\gamma+ h)}\left (1 + \frac{e^{-x h/2}}{c(h)} \right)
\end{equation}
for all $h > 0$. Therefore $\limsup_{x\rightarrow \infty} F(x)/F(x+y) \leq  e^{y (\gamma+ h)} $, and since $h$ is arbitrary  $\limsup_{x\rightarrow \infty} F(x)/F(x+y) \leq  e^{y \gamma} $. But on the other hand also $F(x)/F(x+y) \geq e^{y \gamma}$ for all $x, y \geq 0$ by the Bernstein representation, so we conclude $\lim_{x\rightarrow \infty} F(x)/F(x+y) =  e^{y \gamma} $ for all $y > 0$ which yields the claim.

\smallskip

To prove $(ii)$ we need to show that
$\frac{F(x+y)}{F(x)}$  
 is increasing  in $x$ for all $y \geq 0$. Applying the quotient rule we have  
\begin{equation}
\frac{d}{dx} \frac{ F (x+y)}{F(x)} =\frac{F(x)F'(x+y)-F'( x)F(x+y)}{F( x)^2}.  
\end{equation}
Such an expression is positive if and only if
\begin{equation}\label{eq:logconc}
\frac{F' (x+y)}{F (x+y)} \geq  \frac{F' (x)}{F(x)}, \quad y>0 
\end{equation}
which is in turn equivalent to $\frac{d}{dx}\log F(x)$ being increasing.  But being $F$  c.m. it is also log-convex  (\citet{ste+vh:03}, A.3, Proposition 3.8) and $(ii)$ follows. 
\smallskip


\end{proof}

From Proposition \ref{lemma}, $(i)$, it follows in particular that $\mu \in \L(0)$ if and only if $0 \in \mbox{{\upshape supp}}(\eta)$. Notice that $(ii)$ does not hold when $F$ is not log-convex: for example, if $F_\mu$ is log-concave \eqref{eq:logconc} shows that such convergence is decreasing. An example is  the Weibull distribution i.e. $F_\mu(x)=e^{-(x/\lambda)^k}$, with $k>1, \lambda>0$.

\smallskip

We can now prove convolution equivalence of TS$_\alpha$ distributions, and explicitly identify the tail parameter. The following Theorem  generalizes \citet{kuc+tap:08}, Lemmas 7.3 to 7.5, to arbitrary stability parameters and tempering functions.

\begin{thm}\label{prop:tailsinfty}
Let $\mu$ be a  proper {\upshape TS}$_\alpha$ distribution. Then $\mu \in \mathcal S(\gamma)$ where $\gamma = \inf_{[0,\infty)} \mbox{{\upshape supp}}(Q_+)$.
\end{thm}

\begin{proof}

 We drop the $+$ subscripts. If $\nu$ is the  L\'evy measure of $\mu$, let $\overline \nu(x)=\nu([x, \infty))$, $x \in \mathbb R$. For $x \in \mathbb R$ define the probability density
  \begin{equation}\label{eq:normlaw}
\nu_1(x):=\frac{ \nu(x)  }{\overline \nu(1)}\I_{\{x>1\}}=\frac{\delta}{\overline \nu(1)} \frac{q(x)}{ x^{1+\alpha}} \I_{\{x>1\}}.
\end{equation}
 According to \citet{wat:08}, Theorem B (whose original statement is due to \citet{pakes2004convolution}), we know  $\mu \in \mathcal S(\gamma)$ if and only if $\nu_1 \in \mathcal S(\gamma)$, and  we will prove the latter.


 First we have to show $\nu_1 \in \mathcal L(\gamma)$. 
Reasoning on the densities, for all  $y >0$ it does hold
\begin{align}\label{eq:longtailedTGS}
& \lim_{x \rightarrow \infty} \frac{\overline{ \nu_1}(x+y)}{\overline{ \nu_1}(x)}= \lim_{x\rightarrow \infty}\left( \frac{x}{x+y} \right)^{1+\alpha}  \frac{q(x+y)} {q(x)}=e^{-\gamma y}
\end{align}
having used Proposition \ref{lemma}, $(i)$, on the probability distribution  induced by the c.m. function $q$, which is a s.f. being $\mu$ proper.  

Next, interpreting the convolution in \eqref{eq:yammain} on the L\'evy densities we obtain
\begin{align}\label{tailslim}
\lim_{x\rightarrow \infty}\frac{\overline{\nu_1* \nu_1}(x)}{\overline{\nu_1}(x)}&= \lim_{x\rightarrow \infty}  \frac{1}{\nu_1(x)} \displaystyle{ \int_0^x  \nu_1(x-z)\nu_1(z) dz}  \nonumber \\ &=\frac{1}{\overline \nu(1)}\lim_{x\rightarrow \infty} \frac{1}{\nu_1(x)} \displaystyle{ \int_0^{x}\I_{\{z\geq 1\}}\I_{\{ x-z \geq 1 \}}  \nu(x-z)\nu(z) dz}\nonumber \\ &= \frac{ \delta}{\overline \nu(1)} \lim_{x\rightarrow \infty}\int_{1}^{x-1} \left(\frac{x}{(x-z)z}\right)^{1+\alpha} \frac{q(x-z)q(z)}{q(x)} dz .
\end{align}
Using that the integrand in the last line of \eqref{tailslim} is symmetric in $z$ and $x-z$ about $x/2$ leads to
\begin{align}\label{sim}\int_{1}^{x-1} \left(\frac{x}{(x-z)z}\right)^{1+\alpha} \frac{q(x-z)q(z)}{q(x)} dz  = &2 \int_{1}^{x/2} \left(\frac{x}{(x-z)z}\right)^{1+\alpha} \frac{q(x-z)q(z)}{q(x)} dz .
 \end{align}
Now for  $z \in (1,x/2)$ we have $x/(x-z) \leq 2$. Moreover, by Proposition \ref{lemma}, $(ii)$, for all fixed $ z \geq 1$ the ratio $q(x-z)/q(x)$ is decreasing in $x$ for $x >z$  so that $q(x-z)/q(x)\I_{\{1 \leq z \leq x/2\}}$ has a maximum in $x=2 z$. This produces the estimate
\begin{equation}\label{eq:int}
\left( \frac{x} {(x-z)z}\right)^{1+\alpha}\frac{q(x-z)q(z)}{q(x)}\I_{\{1 \leq z \leq x/2 \} } < 
\left(\frac{2} {z}\right)^{1+\alpha} \frac{q(z)^2}{q(2z)}, \qquad z > 0, \, x >0. \end{equation}
To show the integrability of the right hand side of  \eqref{eq:int}, let $q^\gamma(z):=e^{\gamma z}q(z)$. Since  $q(z)=\mathscr {L}(z, Q)$ then $q^\gamma(z)=\mathscr {L}(z, Q^\gamma)$, where $Q^\gamma$ is the full-support shift of $Q$, i.e. $Q^{\gamma}([0,x))=Q([\gamma, x+\gamma))$, $x \geq 0$. Therefore $q^\gamma$ is c.m. and hence decreasing. Thus as $z\rightarrow \infty$ we have $e^{\gamma z}q(z) \rightarrow c$, $c \geq 0 $, and hence 
\begin{equation}
\frac{q(z)^2}{q(2z)}=\frac{ (e^{\gamma z}q(z))^2}{e^{2 \gamma z}q(2z)} \rightarrow c, \qquad z \rightarrow \infty.
\end{equation}
This implicates that the right hand side of  \eqref{eq:int} is $O(z^{-1-\alpha})$ which is integrable for all	 $\alpha \in (0,2)$.
 We can therefore apply the dominated convergence theorem and continue \eqref{tailslim}  by taking the limit inside the integral. Recalling  $q(x-z) \sim e^{\gamma z}q(x)$ again from Proposition \ref{lemma}, $(i)$, we obtain
\begin{align}\label{eq:final}
\lim_{x \rightarrow \infty}&\frac{\overline{\nu_1*\nu_1}(x)}{\overline{\nu_1}(x)}=\frac{2  \delta}{\overline \nu(1)}  \int_1^{\infty}\lim_{x\rightarrow \infty}  \left( \frac{x}{(x-z)z}\right)^{1+\alpha} \frac{q(x-z) q( z) }{q( x)}\I_{\{1 \leq z \leq x/ 2 \} } dz \nonumber \\&= \frac{2  \delta}{\overline \nu(1)} \int_1^\infty \frac{e^{\gamma z} q(z) }{z^{1+\alpha}} dz= 2 \hat {{\nu_1}}( \gamma). 
\end{align}
This completes the proof that $ \nu_1 \in \mathcal S(\gamma)$, and hence that $\mu \in \mathcal S(\gamma)$.

\end{proof}

Convolution equivalence is a sufficient condition to establish that the L\'evy tail function and the s.f. are of the same order at infinity, from which we can deduce the asymptotics of the tail of a TS$_\alpha$  p.d.f.. Furthermore, the explicit proportionality constant is known.


\begin{cor}\label{cor:tailGTGS}
Let $\mu$ be a proper {\upshape TS}$_\alpha$ distribution 
and let $\gamma_\pm=\inf_{[0,\infty)} \mbox{\upshape supp} (Q_\pm)$ where $Q_\pm$ are the Bernstein representing measures of $q_\pm$. 
Then the p.d.f. $p(x)$ of $\mu$ satisfies
\begin{equation}\label{eq:pasympt}
p(x) \sim  \delta_\pm  \hat \mu^\pm_q( \gamma_\pm) \frac{q_\pm(|x|)}{|x|^{1+\alpha}}, \qquad x \rightarrow \pm \infty.
\end{equation}
where $\mu^\pm_q$ are the probability laws with s.f. $q_\pm$.
\end{cor}
\begin{proof}
When $x \rightarrow \infty$ the claim follows  by Theorem \ref{prop:tailsinfty}, together with Theorem B in \citet{wat:08}, and differentiation of the s.f.s. The case $x \rightarrow -\infty$ is obtained considering the TS$_\alpha$ distribution  with  s.f. $F_\mu(-x)$, whose L\'evy measure is obtained by interchanging $q_+$ with $q_-$, and $\delta_+$ with $\delta_-$ in \eqref{eq:rosrep}, and applying again Theorem \ref{prop:tailsinfty}.
\end{proof}

The key counterexample is a Gamma $\Gamma(a,b)$ distribution with shape $a$ and rate $b$. Then it is known that its L\'evy measure is of the form \eqref{eq:rosrep} with $q_+(x)=e^{-b x}$, $\delta_+=a$, $\delta_-=0$, but with $\alpha=0$. Hence a Gamma law is not a TS$_\alpha$ law. In fact, Gamma laws are not   convolution equivalent: accordingly   $b^{a} x^{a-1}e^{-bx} \Gamma(a)^{-1}  \not \sim a C x^{-1}e^{-bx}$, $x, C,a, b>0$.
  We discuss some explicit applications of Corollary \ref{cor:tailGTGS} in the last section.


\section{Some examples}

For many  TS$_\alpha$ laws in the literature $\hat \mu_q^\pm$ can be determined analytically, which in light of  \ref{cor:tailGTGS}  makes it possible to determine explicit leading orders of the probability density tails.

The most popular TS$_\alpha$ model is the one with exponential tempering,
i.e.~where $q_\pm(x)=e^{-\theta_\pm x}\I_{\{x>0 \}}$, $\theta_\pm >0$. This idea was originally suggested
by~\citet{kop:95} and found widespread application in finance,
e.g.~\citet{boy+lev:00} and~\citet{car+al:02}. A thorough analytical
investigation -- limited to the case $\alpha \in (0,1)$ -- is found
in~\citet{kuc+tap:13}, which we refer to for more details.
In this case $\mu^\pm_q$ are exponential distributions of parameter
$\theta_\pm$ and $\mu^\pm_q \in \L(\theta_\pm)$. The values of the m.g.f.s $\hat {\mu}^\pm_q$ are
well-known in the literature and take different forms in the case $\alpha=1$,
and $\alpha\neq1$  (\citet{con+tan:03}).
From \eqref{eq:pasympt} we see that $p(x) \sim C_\pm \delta_\pm e^{-\theta_\pm |x|}/|x|^{1+\alpha}$, $x \rightarrow \pm \infty$ where
$C_\pm=\hat \mu^\pm_q(\theta_\pm)$. This analysis recovers~\citet{kuc+tap:13}, Theorem 7.10. In order
to show such result the authors did prove  \eqref{prop:tailsinfty} in the
simple case of exponential functions for $q_\pm$.

Another example is the KR TS$_\alpha$ law  based on  upper  gamma-incomplete
tempering introduced in~\citet{kim+al:08}. The authors set (in an equivalent
formulation)
\begin{equation}
\label{eq:KRtemp}
q_\pm(x)=  (\alpha+p_\pm)  \Gamma(-\alpha-p_\pm) \Gamma^{*}\left(  -\alpha - p_\pm, \frac{x}{r_\pm} \right)\I_{\{x>0 \}}, \quad  p_\pm >-\alpha, \, r_\pm >0 
\end{equation}
where
\begin{equation}
\label{eq:gammastar}\Gamma^{*}(s,x)=\frac{x^{-s}}{\Gamma(s)} \int_x^\infty e^{- t} t^{s-1}dt=\frac{1}{\Gamma(s)}\int_{1}^\infty e^{- x t} t^{s-1}dt 
\end{equation}
 is the modified upper incomplete Gamma function. Of course $\Gamma(s,x) \rightarrow 0$, for
all $s$ and $x \rightarrow \infty$. Moreover, from the series representation of
the lower incomplete gamma function $\gamma(s,x)$ (\citet{abr+ste:64}, 6.5.29)
and $\gamma(s,x)+\Gamma(s,x)=\Gamma(s)$ it follows  $\Gamma(s,x)x^{-s} \rightarrow -1/s $, $s <0$, $x \rightarrow 0$, so  that
the KR law is proper.
Furthermore on the right hand side of \eqref{eq:gammastar}  we can apply
Theorem 4 of~\citet{mil+sam:01}, to establish complete monotonicity. 
Then, from 
$\Gamma(s,x)\sim e^{-x}x^{s-1}$, $x \rightarrow \infty$ (\citet{abr+ste:64}, 6.5.32)
applied to \eqref{eq:KRtemp} we obtain
\begin{equation}
\frac{ q_\pm(x+y)}{q_\pm(x)} \sim \left(\frac{x}{x+y}\right) e^{-y/r_\pm} \rightarrow e^{-y/r_\pm}
\end{equation}
as $x \rightarrow \infty$ and thus $\mu^\pm_q \in \L(1/r_\pm)$.   
Applying Corollary \eqref{cor:tailGTGS} and the asymptotics above then shows
\begin{equation}p(x) \sim \delta_\pm   (\alpha + p_\pm) r_\pm C_\pm \frac{e^{-|x|/r_\pm}}{ |x|^{2+\alpha}}, \qquad x \rightarrow \pm \infty 
\end{equation}
where $C_\pm=\hat \mu^\pm_q(1/r_\pm)$ can be explicitly calculated from~\citet{kim+al:08},
Theorem 3.4.

Finally a recent TS$_\alpha$ suggestion  in~\citet{GTGS}  takes into account
the following tempering functions
\begin{equation}
\label{eq:MLtemper}
q_\pm(x)=e^{-\theta_\pm x}E_{\gamma_\pm}(-\lambda_\pm x^{\gamma_\pm})\I_{\{x>0 \}}, \qquad \lambda_\pm >0, \, \theta_\pm \geq 0, \, \gamma_\pm \in (0,1).
\end{equation}
Here
\begin{equation}
E_\gamma(z)=\sum_{k=0}^\infty\frac{z^{k}}{\Gamma(k+\gamma)}, \qquad z \in \mathbb C, \quad \gamma \geq 0
\end{equation}
is the Mittag-Leffler function, a type of heavy-tailed exponential
which is ubiquitous in fractional calculus. It is known that $E_\gamma(-x)$ for
$\gamma \in [0,1]$ is c.m. (\citet{pol:48}). The functions $q_\pm$ are then c.m.
being product of c.m. functions, one of which attained as composition of a c.m.
function with a third one having c.m. derivative. That $q(0)=1, q(\infty)=0$ is clear.
The c.f. of such laws is given in~\citet{GTGS}, Section 3, Theorems 3 and 4.
Furthermore, it is well-known 
 (e.g.~\citet{hau+al:11}) that
\begin{equation}
\label{eq:MLasymptInf}
E_\gamma(- x^\gamma) \sim \frac{x^{-\gamma}}{\Gamma(1-\gamma) }, \quad x   \rightarrow \infty, \qquad \gamma \in (0,1). 
\end{equation}
From the above it follows that  for all $y>0$
\begin{equation}
\frac{q_\pm(x+y)}{q_\pm(x)}\sim e^{-\theta_\pm y} \left(\frac{x}{x+y}\right)^{\gamma_\pm} \rightarrow e^{-\theta_\pm y}, \qquad  x \rightarrow \infty,
\end{equation}
so that $\mu^\pm_q \in \L(\theta_\pm)$. The tilted Mittag-Leffler laws $\mu^\pm_q$ are discussed
in~\citet{tor+al:21}.  Combining  \eqref{eq:MLasymptInf} with
\eqref{eq:pasympt} we are led to the asymptotics
\begin{equation}
\label{eq:infttail}
p(x) \sim C_\pm \frac{\delta_\pm}{\lambda_\pm \Gamma(1-\gamma_\pm)}  \frac{e^{-\theta_\pm |x|}}{|x|^{1+\gamma_\pm+\alpha}}, \qquad x \rightarrow \pm \infty.
\end{equation}
If $\theta_\pm>0$, then $C_\pm=\hat \mu^\pm_q( \theta_\pm)$ is given by~\citet{GTGS}, Theorem 3,
whereas if $\theta_\pm=0$ we have $C_\pm=1$.

The conclusion is that  mixed exponential/Mittag-Leffler tempered
distributions, termed \emph{generalized tempered geometric} stable (GTGS)
in~\citet{GTGS} -- because the radial part of the L\'evy measure extends that
of a geometric tempered (or ``tempered Linnik'') distribution
of~\citet{bar+al:16a} -- retain  exponential tails if and only if $\theta_\pm>0$.
Instead when $\theta_\pm=0$, which corresponds to pure Mittag-Leffler tempering
in \eqref{eq:MLtemper}, we see that  $\mu$ is  heavy (and long)-tailed,
which is unlike the previous two examples. In particular the expectation and
variance are finite if and only if respectively $\alpha+\gamma_\pm >1$ and $\alpha+\gamma_\pm >2$.
In conclusion GTGS laws  capture power law tails while retaining analytical
expressions, which could be useful for e.g.~analyzing financial data. See the
discussion in~\citet{GTGS} for further details.

\bibliographystyle{apalike}
\bibliography{Bibliography}

\end{document}